\documentclass[10pt]{article}

\usepackage{amssymb}
\usepackage{amsfonts}
\usepackage{amsthm}
\usepackage{amsmath}
\usepackage{float}
\usepackage{bigints}
\usepackage{fullpage}

\newtheorem{Lemma}{Lemma}
\newtheorem{Sublemma}{Sublemma}
\newtheorem{Theorem}{Theorem}
\newtheorem{Proposition}{Proposition}

\newtheorem{Definition}{Definition}

\newtheorem{Corollary}{Corollary}
\newtheorem{Remark}{Remark}

\numberwithin{Subcase}{Case}

\usepackage[pdftex]{graphicx}
\graphicspath{{C:/Users/Sam/Desktop/PaperFigures/}}
\DeclareGraphicsExtensions{.pdf}

\title{On contact numbers of totally separable unit sphere packings  
}

\author{K\'{a}roly Bezdek\thanks{Partially supported by a Natural Sciences and 
Engineering Research Council of Canada Discovery Grant.}
, Bal\'azs Szalkai and Istv\'an Szalkai
}

\begin{document}

\maketitle

\begin{abstract}
Contact numbers are natural extensions of kissing numbers. In this paper we give estimates for the number of contacts in a totally separable packing of $n$ unit balls in Euclidean $d$-space for all $n>1$ and $d>1$. 

\end{abstract}

\section{Introduction}
Let $\mathbb{E}^{d}$ denote $d$-dimensional Euclidean space. Then the {\it contact graph} of an arbitrary finite packing of unit balls (i.e., of an arbitrary finite family of closed balls having unit radii and pairwise disjoint interiors) in $\mathbb{E}^{d}$ is the (simple) graph whose vertices correspond to the packing elements and whose two vertices are connected by an edge if and only if the corresponding two packing elements touch each other. The number of edges of a contact graph is called the {\it contact number} of the given unit ball packing. One of the most basic questions on contact graphs is to find the maximum number of edges that a contact graph of a packing of $n$ unit balls can have in $\mathbb{E}^{d}$. Harborth \cite{Ha} proved the following optimal result in $\mathbb{E}^{2}$: the maximum contact number of a packing of $n$ unit disks in $\mathbb{E}^{2}$ is $\lfloor 3n-\sqrt{12n-3}\rfloor$, where$\lfloor\cdot\rfloor$ denotes the lower integer part of the given real. In dimensions three and higher the following upper bounds are known for the maximum contact numbers. It was proved in \cite{BR13} that the contact number of an arbitrary packing of $n$ unit balls in $\mathbb{E}^{3}$ is always less than $6n-0.926n^{\frac{2}{3}}$. On the other hand, it is proved in \cite{B02} that for $d\ge4$ the contact number of an arbitrary packing of $n$ unit balls in $\mathbb{E}^{d}$ is less than
$\frac{1}{2}\tau_d\, n-\frac{1}{2^{d}}\delta_d^{-\frac{d-1}{d}}\; n^{\frac{d-1}{d}}$,
where $\tau_d$ stands for the kissing number of a unit ball in $\mathbb{E}^{d}$ (meaning the maximum number of non-overlapping unit balls of $\mathbb{E}^{d}$ that can touch a given unit ball in $\mathbb{E}^{d}$) and $\delta_d$ denotes the largest possible density for (infinite) packings of unit balls in $\mathbb{E}^{d}$. For further results on contact numbers, including some optimal configurations of packings of small number of unit balls in $\mathbb{E}^{3}$, we refer the interested reader to \cite{ArMaBr11} and \cite{HoHaHe12}. (See also the relevant section in \cite{B13}.)  On the other hand,  \cite{HK01} offers a focused survey on recognition-complexity results of ball contact graphs. For an overview on sphere packings we refer the interested reader to the recent books \cite{B13} and \cite{H12}.

In this paper we investigate the contact numbers of finite unit ball packings that are totally separable. The notion of total separability was introduced in \cite{FTFT73} as follows: a packing of unit balls in $\mathbb{E}^{d}$ is called {\it totally separable} if any two unit balls can be separated by a hyperplane of $\mathbb{E}^{d}$ such that it is disjoint from the interior of each unit ball in the packing. Finding the densest totally separable unit ball packings is a difficult problem, which is solved only in dimensions two (\cite{FTFT73}, \cite{BA83}) and three (\cite{K88}). As a close combinatorial relative we want to investigate the maximum contact number $c(n,d)$ of totally separable packings of $n>1$ unit balls in $\mathbb{E}^{d}, d\ge2$. Before we state our results we make the following observation.  Let $\mathbf{B}^d$ be a unit ball in an arbitrary totally separable packing of unit balls in $\mathbb{E}^{d}$ and assume that $\mathbf{B}^d$ is touched by $m$ unit balls of the given packing say, at the points $\mathbf{t}_1,\dots ,\mathbf{t}_m\in \mathbb{S}^{d-1}$, where the boundary of $\mathbf{B}^d$ is identified with the $(d-1)$-dimensional spherical space $\mathbb{S}^{d-1}$. The total separability of the given packing implies in a straightforward way that the spherical distance between any two points of $\{\mathbf{t}_1,\dots ,\mathbf{t}_m\}$ is at least $\frac{\pi}{2}$. Now, recall that according to \cite{DH51} (see also \cite{Ku07} and \cite{Ra55}) the maximum cardinality of a point set in $\mathbb{S}^{d-1}$ having pairwise spherical distances at least $\frac{\pi}{2}$, is $2d$ and that maximum is attained only for the vertices of a regular $d$-dimensional crosspolytope inscribed in $\mathbf{B}^d$. Thus, $m\le 2d$ and therefore $c(n,d)\le dn$. In the following we state isoperimetric-type improvements on this upper bound. 

A straightforward modification of the method of Harborth \cite{Ha} implies that 
\begin{equation}\label{H-0}
c(n,2)=\lfloor 2n- 2\sqrt{n}\rfloor 
\end{equation}
for all $n>1$. For the convenience of the reader a proof of (\ref{H-0}) is presented in the Appendix of this paper.

Now, let us imagine that we generate totally separable packings of unit diameter balls in $\mathbb{E}^{d}$ such that every center of the balls chosen, is a lattice point of the integer lattice $\mathbb{Z}^{d}$ in $\mathbb{E}^{d}$. Then let $c_{\mathbb{Z}}(n,d)$ denote the largest possible contact number of all totally separable packings of $n$ unit diameter balls obtained in this way. It has been known for a long time (\cite{HaHa}) that $c_{\mathbb{Z}}(n,2)=\lfloor 2n- 2\sqrt{n}\rfloor $, which together with (\ref{H-0}) implies that $c_{\mathbb{Z}}(n,2)=c(n,2)$ for all $n>1$. While we do not know any explicit formula for $c_{\mathbb{Z}}(n,3)$ in terms of $n$, we do have the following simple asymptotic formula for $c_{\mathbb{Z}}(n,3)$ as $n\to+\infty$, which follows in a rather straightforward way from the structural-type theorem of \cite{AC96} characterizing a particular extremal configuration of $c_{\mathbb{Z}}(n,3)$ for any given $n>1$: 
$c_{\mathbb{Z}}(n,3)=3n-3n^{\frac{2}{3}}+o(n^{\frac{2}{3}})$. 
Clearly, $c_{\mathbb{Z}}(n,3)\le c(n,3)$ for all $n>1$. So, one may wonder whether $c_{\mathbb{Z}}(n,3)= c(n,3)$ for all $n>1$?
 
The above discussion leads to the natural and rather basic question on upper bounding $c_{\mathbb{Z}}(n,d)$ (resp., $c(n,d)$) in the form of $dn-Cn^{\frac{d-1}{d}}$, where $C>0$ is a proper constant depending on $d$. 

\begin{Theorem}\label{dD-cubic}
$c_{\mathbb{Z}}(n,d)\le \lfloor dn- d n^{\frac{d-1}{d}}\rfloor $ for all $n>1$ and $d\ge 2$.
\end{Theorem}

We note that the upper bound of Theorem~\ref{dD-cubic} is sharp for $d=2$ and all $n>1$ and for $d\ge 3$ and all $n=k^d$ with $k>1$. On the other hand, it is not a sharp estimate for example, for $d=3$ and $n=5$.

\begin{Theorem}\label{dD}
$c(n,d)\le \bigg\lfloor dn-\frac{1}{2d^{\frac{d-1}{2}}}n^{\frac{d-1}{d}}\bigg\rfloor $ for all $n>1$ and $d\ge 4$. 
\end{Theorem}

Although the method of the proof of Theorem~\ref{dD} can be extended to include the case $d=3$ the following statement is a stronger result.

\begin{Theorem}\label{3D}
$c(n,3)<\lfloor 3n-1.346n^{\frac{2}{3}}\rfloor$ for all $n>1$.
\end{Theorem}

In the rest of the paper we prove the theorems stated.

\section{Proof of Theorem~\ref{dD-cubic}}

A union of finitely many axis parallel $d$-dimensional orthogonal boxes having pairwise disjoint interiors in $\mathbb{E}^{d}$ is called a {\it box-polytope}. One may call the following statement the isoperimetric inequality for box-polytopes, which together with its proof presented below is an analogue of the isoperimetric inequality for convex bodies derived from the Brunn--Minkowski inequality. (For more details on the latter see for example, \cite{Ba97}.) 

\begin{Lemma}\label{isoperimetric-box-polytopes}
Among box-polytopes of given volume the cubes have the least surface volume.
\end{Lemma}

\begin{proof}  Without loss of generality we may assume that the volume ${\rm vol}_d(\mathbf{A})$ of the given box-polytope $\mathbf{A}$
in $\mathbb{E}^{d}$ is equal to $2^d$, i.e., ${\rm vol}_d(\mathbf{A})=2^d$. Let $\mathbf{C}^d$ be an axis parallel $d$-dimensional cube of $\mathbb{E}^{d}$ with ${\rm vol}_d(\mathbf{C}^d)=2^d$. Let the surface volume of $\mathbf{C}^d$ be denoted by ${\rm svol}_{d-1}(\mathbf{C}^d)$. Clearly, ${\rm svol}_{d-1}(\mathbf{C}^d)=d\cdot{\rm vol}_d(\mathbf{C}^d)$. On the other hand,
if ${\rm svol}_{d-1}(\mathbf{A})$ denotes the surface volume of the box-polytope $\mathbf{A}$, then it is rather straightforward to show that

$${\rm svol}_{d-1}(\mathbf{A})=\lim_{ \epsilon\to 0^+}\frac{{\rm vol}_d(\mathbf{A}+\epsilon\mathbf{C}^d)-{\rm vol}_d(\mathbf{A})}{\epsilon}\ ,$$
where $"+"$ in the numerator stands for the Minkowski addition of the given sets. Using the Brunn--Minkowski inequality (\cite{Ba97}) we get that
$${\rm vol}_d(\mathbf{A}+\epsilon\mathbf{C}^d)\ge \left( {\rm vol}_d(\mathbf{A})^{\frac{1}{d}}+{\rm vol}_d(\epsilon\mathbf{C}^d)^{\frac{1}{d}}\right)^d= \left( {\rm vol}_d(\mathbf{A})^{\frac{1}{d}}+\epsilon\cdot {\rm vol}_d(\mathbf{C}^d)^{\frac{1}{d}}\right)^d.$$
Hence,
$${\rm vol}_d(\mathbf{A}+\epsilon\mathbf{C}^d)\ge {\rm vol}_d(\mathbf{A})+d\cdot {\rm vol}_d(\mathbf{A})^{\frac{d-1}{d}}\cdot\epsilon \cdot {\rm vol}_d(\mathbf{C}^d)^{\frac{1}{d}}= {\rm vol}_d(\mathbf{A})+\epsilon \cdot d \cdot {\rm vol}_d(\mathbf{C}^d)= {\rm vol}_d(\mathbf{A})+\epsilon\cdot {\rm svol}_{d-1}(\mathbf{C}^d)\ .  $$
So
$$\frac{{\rm vol}_d(\mathbf{A}+\epsilon\mathbf{C}^d)-{\rm vol}_d(\mathbf{A})}{\epsilon}\ge {\rm svol}_{d-1}(\mathbf{C}^d)$$
and therefore ${\rm svol}_{d-1}(\mathbf{A})\ge {\rm svol}_{d-1}(\mathbf{C}^d)$, finishing the proof of Lemma~\ref{isoperimetric-box-polytopes}. 
\end{proof}

\begin{Corollary}\label{iso-box-poly}
For any box-polytope $\mathbf{P}$ of $\mathbb{E}^{d}$ the isoperimetric quotient $\frac{{\rm svol}_{d-1}(\mathbf{P})^{d}}{{\rm vol}_d(\mathbf{P})^{d-1}}$ of $\mathbf{P}$
is at least as large as the isoperimetric quotient of a cube, i.e., 
$$\frac{{\rm svol}_{d-1}(\mathbf{P})^{d}}{{\rm vol}_d(\mathbf{P})^{d-1}}\ge (2d)^d\ .$$
\end{Corollary}

Now, let $\overline{{\cal P}}:=\{\mathbf{c}_1+\overline{\mathbf{B}}^d, \mathbf{c}_2+\overline{\mathbf{B}}^d, \dots , \mathbf{c}_n+\overline{\mathbf{B}}^d\} $ denote the totally separable packing of $n$ unit diameter balls with centers $\{\mathbf{c}_1, \mathbf{c}_2, \dots , \mathbf{c}_n\}\subset \mathbb{Z}^{d}$ having contact number $c_{\mathbb{Z}}(n,d)$ in $\mathbb{E}^d$. ($\overline{{\cal P}}$ might not be uniquely determined up to congruence in which case $\overline{{\cal P}}$ stands for any of those extremal packings.) Let $\mathbf{U}^d$ be the axis parallel $d$-dimensional unit cube centered at the origin $\mathbf{o}$ in $\mathbb{E}^d$. Then the unit cubes $\{\mathbf{c}_1+\mathbf{U}^d, \mathbf{c}_2+\mathbf{U}^d, \dots , \mathbf{c}_n+\mathbf{U}^d\}$ have pairwise disjoint interiors and $\mathbf{P}=\cup_{i=1}^{n} (\mathbf{c}_i+\mathbf{U}^d)$ is a box-polytope. Clearly, ${\rm svol}_{d-1}(\mathbf{P})=2dn-2c_{\mathbb{Z}}(n,d)$. Hence,
Corollary~\ref{iso-box-poly} implies that
$$2dn-2c_{\mathbb{Z}}(n,d)={\rm svol}_{d-1}(\mathbf{P})  \ge 2d {\rm vol}_d(\mathbf{P})^{\frac{d-1}{d}}=  2dn^{\frac{d-1}{d}}\ .$$
So, $dn-dn^{\frac{d-1}{d}}\ge c_{\mathbb{Z}}(n,d)$, finishing the proof of Theorem~\ref{dD-cubic}.

\section{Proof of Theorem~\ref{dD}}

\begin{Definition}
Let $\mathbf{B}^d=\{\mathbf{x}\in \mathbb{E}^d\ |\ \|\mathbf{x}\|\le 1\}$ be the closed unit ball centered at the origin $\mathbf{o}$ in $\mathbb{E}^d$, where $\|\cdot\|$ refers to the standard Euclidean norm of $\mathbb{E}^d$. Let $R\ge 1$. We say that the packing 
$${\cal P}_{{\rm sep}}=\{\mathbf{c}_i+\mathbf{B}^d\ |\ i\in I \ {\rm with}\ \| \mathbf{c}_j-\mathbf{c}_k\|\ge 2 \ {\rm for\ all}\ j\neq k\in I\}$$ 
of (finitely or infinitely many) non-overlapping translates of $\mathbf{B}^d$ with centers $\{\mathbf{c}_i\ |\ i\in I\}$ is an {\rm $R$-separable packing} in $\mathbb{E}^d$ if for each $i\in I$ the finite packing $\{\mathbf{c}_j+\mathbf{B}^d\ |\ \mathbf{c}_j+\mathbf{B}^d\subseteq \mathbf{c}_i+R\mathbf{B}^d\}$ is a totally separable packing (in $\mathbf{c}_i+R\mathbf{B}^d$). Finally, let $\delta_{{\rm sep}}(R, d)$ denote the largest density of all $R$-separable unit ball packings in $\mathbb{E}^d$, i.e., let
$$\delta_{{\rm sep}}(R, d)=\sup_{{\cal P}_{{\rm sep}}}\left(\limsup_{\lambda\to+\infty}\frac{\sum_{\mathbf{c}_i+\mathbf{B}^d\subset\mathbf{Q}_{\lambda}}{\rm vol}_d(\mathbf{c}_i+\mathbf{B}^d)}{{\rm vol}_d(\mathbf{Q}_{\lambda})}\right)\ , $$
where $\mathbf{Q}_{\lambda}$ denotes the $d$-dimensional cube of edge length $2\lambda$ centered at $\mathbf{o}$ in $\mathbb{E}^d$ having edges parallel to the coordinate axes of $\mathbb{E}^d$.
\end{Definition}

\begin{Remark}
For any $1\le R< 3$ we have that $\delta_{{\rm sep}}(R, d)=\delta_d$, where $\delta_d$ stands for the supremum of the upper densities of all unit ball packings in $\mathbb{E}^d$.
\end{Remark}

The following statement is the core part of our proof of Theorem~\ref{dD} and it is an analogue of the Lemma in \cite{B02} (see also Theorem 3.1 in \cite{BHW}).

\begin{Theorem}\label{R-separable}
If $\{\mathbf{c}_i+\mathbf{B}^d\ |\ 1\le i\le n\}$ is an $R$-separable packing of $n$ unit balls in $\mathbb{E}^d$ with $R\ge 1$, $n\ge 1$, and $d\ge 2$, then
$$\frac{n{\rm vol}_d(\mathbf{B}^d) }{{\rm vol}_d\left(\cup_{i=1}^n  \mathbf{c}_i+2R\mathbf{B}^d\right)}\le\delta_{{\rm sep}}(R, d)\ .$$ 
\end{Theorem}

\begin{proof}
Assume that the claim is not true. Then there is an $\epsilon>0$ such that
\begin{equation}\label{R1}
{\rm vol}_d\left(\cup_{i=1}^n  \mathbf{c}_i+2R\mathbf{B}^d\right)=\frac{n{\rm vol}_d(\mathbf{B}^d) }{\delta_{{\rm sep}}(R, d)}-\epsilon
\end{equation}
Let $C_n=\{\mathbf{c}_i\ |\ i=1,\dots ,n\}$ and let $\Lambda$ be a packing lattice of $C_n+2R\mathbf{B}^d=\cup_{i=1}^n  \mathbf{c}_i+2R\mathbf{B}^d$ such that $C_n+2R\mathbf{B}^d$ is contained in the origin symmetric fundamental parallelotope $\mathbf{P}$ of $\Lambda$. Recall that for each $\lambda>0$, $\mathbf{Q}_{\lambda}$ denotes the $d$-dimensional cube of edge length $2\lambda$ centered at the origin $\mathbf{o}$ in $\mathbb{E}^d$ having edges parallel to the coordinate axes of $\mathbb{E}^d$. Clearly, there is a constant $\mu>0$ depending on $\mathbf{P}$ only, such that for each $\lambda>0$ there is a subset $L_{\lambda}$ of $\Lambda$ with 
\begin{equation}\label{R2}
\mathbf{Q}_{\lambda}\subseteq L_{\lambda}+\mathbf{P}\ {\rm and}\ L_{\lambda}+2\mathbf{P}\subseteq\mathbf{Q}_{\lambda + \mu}\ .
\end{equation}
%Moreover, let ${\cal P}_m(\mathbf{B}^d)$ denote the family of all $R$-separable packings of $m>1$ unit balls in $\mathbb{E}^d$. 
The definition of $\delta_{{\rm sep}}(R, d)$ implies that for each $\lambda>0$ there exists an $R$-separable packing of $m(\lambda)$ translates of $\mathbf{B}^d$ in 
$\mathbb{E}^d$ with centers at the points of $C(\lambda)$ such that
$$C(\lambda)+\mathbf{B}^d\subset \mathbf{Q}_{\lambda}$$
and
$$\lim_{\lambda\to+\infty}\frac{m(\lambda){\rm vol}_d(\mathbf{B}^d)}{{\rm vol}_d( \mathbf{Q}_{\lambda})}=\delta_{{\rm sep}}(R, d)\ .$$
As $\lim_{\lambda\to+\infty}\frac{{\rm vol}_d(\mathbf{Q}_{\lambda+\mu})}{{\rm vol}_d(\mathbf{Q}_{\lambda})}=1$ therefore there exist $\xi>0$ and an $R$-separable packing
of $m(\xi)$ translates of $\mathbf{B}^d$ in $\mathbb{E}^d$ with centers at the points of $C(\xi)$ and with $C(\xi)+\mathbf{B}^d\subset \mathbf{Q}_{\xi}$ such that
\begin{equation}\label{R3}
\frac{{\rm vol}_d(\mathbf{P})\delta_{{\rm sep}}(R, d)}{{\rm vol}_d(\mathbf{P})+\epsilon}<\frac{m(\xi){\rm vol}_d(\mathbf{B}^d)}{{\rm vol}_d( \mathbf{Q}_{\xi+\mu})}
\ {\rm and}\    \frac{n{\rm vol}_d(\mathbf{B}^d)}{{\rm vol}_d(\mathbf{P})+\epsilon}<\frac{n{\rm vol}_d(\mathbf{B}^d){\rm card}(L_{\xi})}{{\rm vol}_d(\mathbf{Q}_{\xi+\mu})}\ ,
\end{equation}
where ${\rm card}(\cdot)$ refers to the cardinality of the given set.
Now, for each $\mathbf{x}\in\mathbf{P}$ we define an $R$-separable packing of $\overline{m}(\mathbf{x})$ translates of $\mathbf{B}^d$ in $\mathbb{E}^d$ with centers at the points of
$$\overline{C}(\mathbf{x})=\{\mathbf{x}+L_{\xi}+C_n\}\cup\{\mathbf{y}\in C(\xi)\ |\ \mathbf{y}\notin \mathbf{x}+L_{\xi}+C_n+{\rm int}(2R\mathbf{B}^d)\}\ ,  $$
where ${\rm int}(\cdot)$ refers to the interior of the given set in $\mathbb{E}^d$. Clearly, (\ref{R2}) implies that $\overline{C}(\mathbf{x})+\mathbf{B}^d\subset\mathbf{Q}_{\xi+\mu}$.
Now, in order to evaluate $\int_{\mathbf{x}\in\mathbf{P}}\overline{m}(\mathbf{x})d\mathbf{x}$, we introduce the function $\chi_{\mathbf{y}}$ for each $\mathbf{y}\in C(\xi)$ defined as follows: $\chi_{\mathbf{y}}(\mathbf{x})=1$ if $\mathbf{y}\notin \mathbf{x}+L_{\xi}+C_n+{\rm int}(2R\mathbf{B}^d)$ and $\chi_{\mathbf{y}}(\mathbf{x})=0$ for any other $\mathbf{x}\in\mathbf{P}$. Based on the origin symmetric $\mathbf{P}$ it is easy to see that for any  $\mathbf{y}\in C(\xi)$ one has $\int_{\mathbf{x}\in\mathbf{P}}
\chi_{\mathbf{y}}(\mathbf{x})d\mathbf{x}={\rm vol}_d(\mathbf{P})-{\rm vol}_d(C_n+2R\mathbf{B}^d)$. Thus, it follows in a straightforward way that

$$\int_{\mathbf{x}\in\mathbf{P}}\overline{m}(\mathbf{x})d\mathbf{x}=\int_{\mathbf{x}\in\mathbf{P}}\big(n{\rm card}(L_{\xi})+\sum_{\mathbf{y}\in C(\xi)}\chi_{\mathbf{y}}(\mathbf{x})\big)d\mathbf{x}=n{\rm vol}_d(\mathbf{P}){\rm card}(L_{\xi})+m(\xi)\big({\rm vol}_d(\mathbf{P})-{\rm vol}_d(C_n+2R\mathbf{B}^d)\big)\ .$$ 
Hence, there is a point $\mathbf{p}\in\mathbf{P}$ with
$$\overline{m}(\mathbf{p})\ge m(\xi)\left(1-\frac{{\rm vol}_d(C_n+2R\mathbf{B}^d)}{{\rm vol}_d(\mathbf{P})}\right)+n{\rm card}(L_{\xi})$$
and so
\begin{equation}\label{R4}
\frac{\overline{m}(\mathbf{p}){\rm vol}_d(\mathbf{B}^d)}{{\rm vol}_d(\mathbf{Q}_{\xi+\mu})}\ge \frac{m(\xi){\rm vol}_d(\mathbf{B}^d)}{{\rm vol}_d(\mathbf{Q}_{\xi+\mu})}\left(1-\frac{{\rm vol}_d(C_n+2R\mathbf{B}^d)}{{\rm vol}_d(\mathbf{P})}\right)+\frac{n{\rm vol}_d(\mathbf{B}^d){\rm card}(L_{\xi})}{{\rm vol}_d(\mathbf{Q}_{\xi+\mu})}\ .
\end{equation}
Now, (\ref{R1}) implies in a straightforward way that
\begin{equation}\label{R5}
\frac{{\rm vol}_d(\mathbf{P}) \delta_{{\rm sep}}(R, d) }{{\rm vol}_d(\mathbf{P})+\epsilon}\left(1-\frac{{\rm vol}_d(C_n+2R\mathbf{B}^d)}{{\rm vol}_d(\mathbf{P})}\right)+\frac{n{\rm vol}_d(\mathbf{B}^d)}{{\rm vol}_d(\mathbf{P})+\epsilon}= \delta_{{\rm sep}}(R, d)
\end{equation}
Thus, (\ref{R3}), (\ref{R4}), and (\ref{R5}) yield that
$$\frac{\overline{m}(\mathbf{p}){\rm vol}_d(\mathbf{B}^d)}{{\rm vol}_d(\mathbf{Q}_{\xi+\mu})}>\delta_{{\rm sep}}(R, d)\ .$$
As $\overline{C}(\mathbf{p})+\mathbf{B}^d\subset\mathbf{Q}_{\xi+\mu}$ this contradicts the definition of $\delta_{{\rm sep}}(R, d)$, finishing the proof of Theorem~\ref{R-separable}.
\end{proof}

Next, let ${\cal P}=\{\mathbf{c}_1+\mathbf{B}^d, \mathbf{c}_2+\mathbf{B}^d, \dots , \mathbf{c}_n+\mathbf{B}^d\}$ be a totally separable packing of $n$ translates of $\mathbf{B}^d$
with centers at the points of $C_n=\{\mathbf{c}_1, \mathbf{c}_2, \dots , \mathbf{c}_n\}$ in $\mathbb{E}^d$. Recall that any member of ${\cal P}$ is tangent to at most $2d$ members of ${\cal P}$ and if $\mathbf{c}_i+\mathbf{B}^d$ is tangent to $2d$ members, then the tangent points are the vertices of a regular cross-polytope inscribed in $\mathbf{c}_i+\mathbf{B}^d$ and therefore
$$\mathbf{c}_i+\sqrt{d}\mathbf{B}^d\subset\bigcup_{1\le j\le n, j\neq i}\mathbf{c}_j+\sqrt{d}\mathbf{B}^d\ .$$
Thus, if $m$ denotes the number of members of ${\cal P}$ that are tangent to $2d$ members in ${\cal P}$, then the $(d-1)$-dimensional surface volume ${\rm svol}_{d-1}\left({\rm bd}(C_n+\sqrt{d}\mathbf{B}^d)\right)$ of the boundary ${\rm bd}(C_n+\sqrt{d}\mathbf{B}^d)$ of the non-convex set $C_n+\sqrt{d}\mathbf{B}^d$ must satisfy the inequality
\begin{equation}\label{R6}
{\rm svol}_{d-1}\left({\rm bd}(C_n+\sqrt{d}\mathbf{B}^d)\right)\le (n-m)d^{\frac{d-1}{2}}{\rm svol}_{d-1}\left({\rm bd}(\mathbf{B}^d)\right)\ .
\end{equation}
Finally, the isoperimetric inequality (\cite{Os78}) applied to $C_n+\sqrt{d}\mathbf{B}^d$ yields
\begin{equation}\label{R7}
{\rm Iq}(\mathbf{B}^d)=\frac{{\rm svol}_{d-1}\left({\rm bd}(\mathbf{B}^d)\right)^d}{{\rm vol}_d(\mathbf{B}^d)^{d-1}}=d^d{\rm vol}_d(\mathbf{B}^d)\le {\rm Iq}(C_n+\sqrt{d}\mathbf{B}^d)=\frac{{\rm svol}_{d-1}\left({\rm bd}(C_n+\sqrt{d}\mathbf{B}^d)\right)^d}{{\rm vol}_d(C_n+\sqrt{d}\mathbf{B}^d)^{d-1}}\ ,
\end{equation}
where ${\rm Iq}(\cdot)$ stands for the isoperimetric quotient of the given set. As $d\ge 4$, ${\cal P}$ is a $\frac{\sqrt{d}}{2}$-separable packing (in fact, it is an $R$-separable packing for all $R\ge 1$) and therefore (\ref{R6}), (\ref{R7}), and Theorem~\ref{R-separable} imply in a straightforward way that
$$n-m\ge \frac{{\rm svol}_{d-1}\left({\rm bd}(C_n+\sqrt{d}\mathbf{B}^d)\right)}{d^{\frac{d-1}{2}}{\rm svol}_{d-1}\left({\rm bd}(\mathbf{B}^d)\right)}=\frac{{\rm svol}_{d-1}\left({\rm bd}(C_n+\sqrt{d}\mathbf{B}^d)\right)}{d^{\frac{d+1}{2}}{\rm vol}_{d}(\mathbf{B}^d)}\ge \frac{{\rm Iq}(\mathbf{B}^d)^{\frac{1}{d}} {\rm vol}_d(C_n+\sqrt{d}\mathbf{B}^d)^{\frac{d-1}{d}} }{d^{\frac{d+1}{2}}{\rm vol}_{d}(\mathbf{B}^d)}$$
$$\ge \frac{{\rm Iq}(\mathbf{B}^d)^{\frac{1}{d}} }{d^{\frac{d+1}{2}}{\rm vol}_{d}(\mathbf{B}^d)}\left(\frac{n{\rm vol}_d(\mathbf{B}^d)}{\delta_{{\rm sep}}(\frac{\sqrt{d}}{2}, d)} \right)^{\frac{d-1}{d}}=\frac{1}{d^{\frac{d-1}{2}} \delta_{{\rm sep}}(\frac{\sqrt{d}}{2}, d)^{\frac{d-1}{d}}}n^{\frac{d-1}{d}}\ .$$
Thus, the number of contacts in ${\cal P}$ is at most
$$\frac{1}{2}\left(2dn-(n-m)\right)\le dn-\frac{1}{2d^{\frac{d-1}{2}} \delta_{{\rm sep}}(\frac{\sqrt{d}}{2}, d)^{\frac{d-1}{d}}}n^{\frac{d-1}{d}}<dn-\frac{1}{2d^{\frac{d-1}{2}} }n^{\frac{d-1}{d}}\ ,$$ finishing the proof of Theorem~\ref{dD}.

\section{Proof of Theorem~\ref{3D}}

The following proof is an analogue of the proof of Theorem 1.1 in \cite{B12} and as such it is based on the proper modifications of the main (resp., technical) lemmas of \cite{B12}. Overall the method discussed below turns out to be more efficient for totally separable unit ball packings than for unit ball packings in general. The more exact details are as follows.
 
Let ${\cal P}:=\{\mathbf{c}_1+\mathbf{B}^3, \mathbf{c}_2+\mathbf{B}^3, \dots , \mathbf{c}_n+\mathbf{B}^3\} $ denote the totally separable packing of $n$ unit balls with centers $\mathbf{c}_1, \mathbf{c}_2, \dots , \mathbf{c}_n$ in $\mathbb{E}^3$, which has the largest number namely, $c(n,3)$ of touching pairs among all totally separable packings of $n$ unit balls in $\mathbb{E}^3$. (${\cal P}$ might not be uniquely determined up to congruence in which case ${\cal P}$ stands for any of those extremal packings.)

\begin{Lemma}\label{3D density upper bound}
$$\frac{\frac{4\pi}{3}n}{{\rm vol}_3\left(\bigcup_{i=1}^n\left(\mathbf{c}_i+\sqrt{3}\mathbf{B}^3\right)\right)}<0.6401,$$
where ${\rm vol}_3(\cdot)$ refers to the $3$-dimensional volume of the corresponding set.
\end{Lemma}

\begin{proof}       
First, partition $\bigcup_{i=1}^n\left(\mathbf{c}_i+\sqrt{3}\mathbf{B}^3\right)$ into truncated Voronoi cells as follows. Let $\mathbf{P}_i$ denote the Voronoi cell of the packing $\cal P$ assigned to $\mathbf{c}_i+\mathbf{B}^3$, $1\le i\le n$, that is, let $\mathbf{P}_i$ stand for the set of points of $\mathbb{E}^{3}$ that are not farther away from $\mathbf{c}_i$ than from any other $\mathbf{c}_j$ with $j\neq i, 1\le j\le n$. Then, recall the well-known fact (see for example, \cite{FT64}) that the Voronoi cells $\mathbf{P}_i$, $1\le i\le n$ just introduced form a tiling of $\mathbb{E}^{3}$. Based on this it is easy to see that the truncated Voronoi cells $\mathbf{P}_i\cap (\mathbf{c}_i+\sqrt{3}\mathbf{B}^3)$, $1\le i\le n$ generate a tiling of the non-convex container $\bigcup_{i=1}^n\left(\mathbf{c}_i+\sqrt{3}\mathbf{B}^3\right)$ for the packing $\cal P$. Second, we prove the following metric properties of the Voronoi cells introduced above.

\begin{Sublemma}\label{metric-1}
The distance between the line of an arbitrary edge of the Voronoi cell $\mathbf{P}_i$ and the center $\mathbf{c}_i$ is always at least $\frac{3\sqrt{3}}{4}=1.299\dots$ for any $1\le i\le n$.
\end{Sublemma}

\begin{proof}
It is easy to see that the claim follows from the following $2$-dimensional statement: If $\{\mathbf{a}+\mathbf{B}^2, \mathbf{b}+\mathbf{B}^2, \mathbf{c}+\mathbf{B}^2\} $ is a totally separable packing of three unit disks with centers $\mathbf{a}, \mathbf{b}, \mathbf{c}$ in $\mathbb{E}^2$, then the circumradius of the triangle $\mathbf{a}\mathbf{b}\mathbf{c}$ is at least $\frac{3\sqrt{3}}{4}$. We prove the latter statement as follows. If some side of the triangle $\mathbf{a}\mathbf{b}\mathbf{c}$ has length at least $2\sqrt{2}$, then the circumradius of the triangle $\mathbf{a}\mathbf{b}\mathbf{c}$ is at least $\sqrt{2}>\frac{3\sqrt{3}}{4}=1.299\dots$. So, without loss of generality we may assume that 
$2<\|\mathbf{a}-\mathbf{b}\|<2\sqrt{2}$, $2<\|\mathbf{a}-\mathbf{c}\|<2\sqrt{2}$, and $2<\|\mathbf{b}-\mathbf{c}\| <2\sqrt{2}$ and so $\mathbf{a}\mathbf{b}\mathbf{c}$
is an acute triangle. Moreover, as the three unit disks with centers  $\mathbf{a}, \mathbf{b}, \mathbf{c}$ form a totally separable packing therefore there must exist two unit disks say, $\mathbf{a}+\mathbf{B}^2$ and $\mathbf{b}+\mathbf{B}^2$ such that their two inner tangent lines are disjoint from the interior of the third unit disk $\mathbf{c}+\mathbf{B}^2$ separating the unit disks $\mathbf{a}+\mathbf{B}^2$,  $\mathbf{c}+\mathbf{B}^2$ (resp., $\mathbf{b}+\mathbf{B}^2$,  $\mathbf{c}+\mathbf{B}^2$) from $\mathbf{b}+\mathbf{B}^2$ (resp., $\mathbf{a}+\mathbf{B}^2$). Finally, if necessary then by properly translating $\mathbf{c}+\mathbf{B}^2$ and thereby decreasing the circumradius of the triangle $\mathbf{a}\mathbf{b}\mathbf{c}$ one can assume that the two inner tangent lines of the unit disks $\mathbf{a}+\mathbf{B}^2$ and $\mathbf{b}+\mathbf{B}^2$ are tangent to the unit disk $\mathbf{c}+\mathbf{B}^2$ with $2<\|\mathbf{a}-\mathbf{b}\|<2\sqrt{2}$ and $2<\|\mathbf{a}-\mathbf{c}\|=\|\mathbf{b}-\mathbf{c}\| <2\sqrt{2}$. Now, if $x=\frac{1}{2}\|\mathbf{a}-\mathbf{b}\|$, then an elementary computation yields that the circumradius of the triangle $\mathbf{a}\mathbf{b}\mathbf{c}$ is equal to $f(x)=\frac{x^3}{2\sqrt{x^2-1}}$ with $1<x<\sqrt{2}$. Hence, $f'(x)=\frac{x^2(2x^2-3)}{2(x^2-1)\sqrt{x^2-1}}$ implies in a straightforward way that $f(\sqrt{\frac{3}{2}})=\frac{3\sqrt{3}}{4}$ is a global minimum of $f(x)$ over $1<x<\sqrt{2}$. This finishes the proof of Sublemma~\ref{metric-1}.
\end{proof}

\begin{Remark}
As one can see from the above proof, the lower bound of Sublemma~\ref{metric-1} is a sharp one and it should be compared to the lower bound $\frac{2}{\sqrt{3}}=1.154\dots$ valid for any unit ball packing not necessarily totally separable in $\mathbb{E}^{3}$. (For more details on the lower bound $\frac{2}{\sqrt{3}}$ see for example the discussion on page 29 in \cite{B13}.)
\end{Remark}

\begin{Sublemma}\label{metric-2}
The distance between an arbitrary vertex of the Voronoi cell $\mathbf{P}_i$ and the center $\mathbf{c}_i$ is always at least $\sqrt{2}=1.414\dots$ for any $1\le i\le n$.
\end{Sublemma}
\begin{proof}
Clearly, the claim follows from the following statement: If ${\cal P}_4=\{\mathbf{c}_1+\mathbf{B}^3, \mathbf{c}_2+\mathbf{B}^3, \mathbf{c}_3+\mathbf{B}^3,  \mathbf{c}_4+\mathbf{B}^3\} $ is a totally separable packing of four unit balls with centers $\mathbf{c}_1, \mathbf{c}_2, \mathbf{c}_3, \mathbf{c}_4$ in $\mathbb{E}^3$, then the circumradius of the terahedron $\mathbf{c}_1\mathbf{c}_2\mathbf{c}_3\mathbf{c}_4$ is at least $\sqrt{2}$. We prove the latter claim by looking at the following two possible cases. ${\cal P}_4$ is a totally separable packing with plane $H$ separating either $\mathbf{c}_1+\mathbf{B}^3, \mathbf{c}_2+\mathbf{B}^3$ from $ \mathbf{c}_3+\mathbf{B}^3,  \mathbf{c}_4+\mathbf{B}^3$ ({\it Case 1}) or $\mathbf{c}_1+\mathbf{B}^3$ from $\mathbf{c}_2+\mathbf{B}^3, \mathbf{c}_3+\mathbf{B}^3,  \mathbf{c}_4+\mathbf{B}^3$ ({\it Case 2}). In both cases it is sufficient to show that if $\cup_{i=1}^4\mathbf{c}_i+\mathbf{B}^3\subset\mathbf{x}+r\mathbf{B}^3$ for some $\mathbf{x}\in\mathbb{E}^3$ and $r\in\mathbb{R}$, then $r\ge 1+\sqrt{2}$.

\noindent{\it Case 1}: Let $H^+$ and $H^-$ denote the two closed halfspaces bounded by $H$ with $\mathbf{c}_1+\mathbf{B}^3\cup \mathbf{c}_2+\mathbf{B}^3\subset H^+$ and $\mathbf{c}_3+\mathbf{B}^3\cup  \mathbf{c}_4+\mathbf{B}^3\subset H^-$. Without loss of generality, we may assume that $\mathbf{x}\in H^-$. Now, if $\mathbf{c}_1'$ (resp., $\mathbf{c}_2'$) denotes the image of $\mathbf{c}_1$ (resp., $\mathbf{c}_2$) under the reflection in $H$, then clearly ${\cal P}'=\{\mathbf{c}_1+\mathbf{B}^3, \mathbf{c}_2+\mathbf{B}^3, \mathbf{c}_1'+\mathbf{B}^3,  \mathbf{c}_2'+\mathbf{B}^3\}$ is a packing of four unit balls in $\mathbf{x}+r\mathbf{B}^3$ symmetric about $H$. Then using the symmetry of ${\cal P}'$ with respect to $H$ it is easy to see that $r\ge 1+\sqrt{2}$.

\noindent{\it Case 2}: Let $H^+$ and $H^-$ denote the two closed halfspaces bounded by $H$ with $\mathbf{c}_1+\mathbf{B}^3\subset H^+$ and $ \mathbf{c}_2+\mathbf{B}^3\cup\mathbf{c}_3+\mathbf{B}^3\cup  \mathbf{c}_4+\mathbf{B}^3\subset H^-$. If one assumes that $r-1<\sqrt{2}$, then using $\mathbf{c}_1\in (\mathbf{x}+(r-1)\mathbf{B}^3)\cap H^+ $ and $\{\mathbf{c}_2, \mathbf{c}_3, \mathbf{c}_4\}\subset (\mathbf{x}+(r-1)\mathbf{B}^3)\cap H^- $ it is easy to see that the triangle $\mathbf{c}_2\mathbf{c}_3\mathbf{c}_4$ is contained in a disk of radius less than $2\sqrt{\sqrt{2}-1}=1.287\dots$.
On the other hand, as the unit balls $ \mathbf{c}_2+\mathbf{B}^3, \mathbf{c}_3+\mathbf{B}^3,  \mathbf{c}_4+\mathbf{B}^3$ form a totally separable packing therefore the proof of Sublemma~\ref{metric-1} implies that the radius of any disk containing the triangle $\mathbf{c}_2\mathbf{c}_3\mathbf{c}_4$ must be at least $\frac{3\sqrt{3}}{4}=1.299\dots$, a contradiction.
\end{proof}
\begin{Remark}
As one can see from the above proof, the lower bound of Sublemma~\ref{metric-2} is a sharp one and it should be compared to the lower bound $\sqrt{\frac{3}{2}}=1.224\dots$ valid for any unit ball packing not necessarily totally separable in $\mathbb{E}^{3}$. (For more details on the lower bound  $\sqrt{\frac{3}{2}}$ see for example the discussion on page 29 in \cite{B13}.)
\end{Remark}

Now, we are ready to prove Sublemma~\ref{Rogers-Bezdek}. As the method used is well-known we give only an outline of the major steps of its proof, which originates from Rogers (\cite{R64}). In fact, what we need here is a truncated version of Rogers's method that has been introduced by B\"or\"oczky in \cite{Bo} (also for spherical and hyperbolic spaces). We recommend the interested reader to look up the relevant details in \cite{Bo}. First we need to recall the notion of an orthoscheme. (In what follows ${\rm conv}\{\cdot\}$ refers to the convex hull of the given set.)
\begin{Definition}\label{orthoschemes}
The $i$-dimensional simplex $Y={\rm conv}\{\mathbf{o}, \mathbf{y}_1,
\dots , \mathbf{y}_i\}\subset \mathbb{E}^{d}$ with vertices $\mathbf{y}_0=\mathbf{o}, \mathbf{y}_1, \dots , \mathbf{y}_i$ is called an 
{\it $i$-dimensional orthoscheme} if for each
$j, 0\le j\le i-1$ the vector $\mathbf{y}_j$ is orthogonal to the linear hull
${\rm lin}\{\mathbf{y}_k-\mathbf{y}_j\ \vert \ j+1\le k\le i\}$, where
$1\le i\le d$.\end{Definition}
\noindent Next we dissect each truncated Voronoi cell $\mathbf{P}_i\cap (\mathbf{c}_i+\sqrt{2}\mathbf{B}^3)$, $1\le i\le n$ into $3$-dimensional, $2$-dimensional, and $1$-dimensional orthoschemes having pairwise disjoint relative interiors as follows. Namely, for each $\mathbf{x}\in \mathbf{P}_i\cap {\rm bd}(\mathbf{c}_i+\sqrt{2}\mathbf{B}^3)$ we assign an orthoscheme in the following well-defined way. (We note that due to Sublemma~\ref{metric-2} the intersection $\mathbf{P}_i\cap {\rm bd}(\mathbf{c}_i+\sqrt{2}\mathbf{B}^3)$ is always non-empty.) We distinguish three cases. If $\mathbf{x}\in {\rm int} \mathbf{P}_i$, then the assigned orthoscheme is the line segment ${\rm conv}\{\mathbf{c}_i, \mathbf{x}\}$. 
If $\mathbf{x}$ is a relative interior point of some face $F$ of $\mathbf{P}_i$, then we assign to $\mathbf{x}$ the orthoscheme ${\rm conv}\{\mathbf{c}_i, \mathbf{f}, \mathbf{x}\}$, where $\mathbf{f}$ is the orthogonal projection of $\mathbf{c}_i$ onto the plane of $F$. (We note that $\mathbf{f}$ lies in $F$). If $\mathbf{x}$ is a (relative interior) point of some edge $E$ of $\mathbf{P}_i$ with $E$ lying on the face $F$ of $\mathbf{P}_i$, then we assign to $\mathbf{x}$ the orthoscheme ${\rm conv}\{\mathbf{c}_i, \mathbf{f}, \mathbf{e}, \mathbf{x}\}$, where $\mathbf{e}$ (resp., $\mathbf{f}$) is the orthogonal projection of $\mathbf{c}_i$ onto the line of $E$ (resp., onto the plane of $F$).  (We note that $\mathbf{e}$ (resp., $\mathbf{f}$) belongs to $E$ (resp., $F$).)  This completes the process of dissecting $\mathbf{P}_i\cap (\mathbf{c}_i+\sqrt{2}\mathbf{B}^3)$ into orthoschemes.

\noindent As a next step we need to recall the so-called Lemma of Comparison of Rogers (for more details see for example, page 33 in \cite{B13}).

\begin{Proposition}\label{Rogers-lemma}
Let $\mathbf{W}:={\rm conv}\{\mathbf{o}, \mathbf{w}_1, \dots , \mathbf{w}_d\}$ be a $d$-dimensional orthoscheme in $\mathbb{E}^{d}$. Moreover, let $\mathbf{U}:= {\rm conv}\{\mathbf{o}, \mathbf{u}_1, \dots , \mathbf{u}_d\}$ be a $d$-dimensional simplex of $\mathbb{E}^{d}$ such that $\|\mathbf{u}_i \|={\rm dist}\left(\mathbf{o}, {\rm conv}\{\mathbf{u}_i, \mathbf{u}_{i+1}, \dots , \mathbf{u}_d\} \right)$ for all $1\le i\le d$, where ${\rm dist}(\cdot , \cdot )$ refers to the usual Euclidean distance between two given sets. If $1\le\|\mathbf{w}_i\|\le\|\mathbf{u}_i\|$ holds for all $1\le i\le d$, then
$$
\frac{{\rm vol}_d(\mathbf{U}\cap\mathbf{B}^d)}{{\rm vol}_d(\mathbf{U})}\le
\frac{{\rm vol}_d(\mathbf{W}\cap\mathbf{B}^d)}{{\rm vol}_d(\mathbf{W})}.
$$ 
\end{Proposition}

\noindent Finally, let $\mathbf{W}^3:={\rm conv}(\{\mathbf{o}, \mathbf{w}_1, \mathbf{w}_2, \mathbf{w}_3\})$ be the $3$-dimensional orthoscheme with $\|\mathbf{w}_1\|=1$, $\|\mathbf{w}_2\|=\frac{3\sqrt{3}}{4}$, and $\|\mathbf{w}_3\|=\sqrt{2}$. Clearly, Sublemmas~\ref{metric-1}, \ref{metric-2}, and Proposition~\ref{Rogers-lemma} imply that for any $3$-dimensional orthoscheme $\mathbf{U}^3:= {\rm conv}\{\mathbf{c}_i, \mathbf{f}, \mathbf{e}, \mathbf{x}\}$ of the above dissection of the trunceted Voronoi cell $\mathbf{P}_i\cap (\mathbf{c}_i+\sqrt{2}\mathbf{B}^3)$ we have that
$$
\frac{{\rm vol}_3(\mathbf{U}^3\cap\mathbf{B}^3)}{{\rm vol}_d(\mathbf{U}^3)}\le
\frac{{\rm vol}_3(\mathbf{W}^3\cap\mathbf{B}^3)}{{\rm vol}_d(\mathbf{W}^3)}.
$$ 

\noindent As each $2$-dimensional (resp., $1$-dimensional) orthoscheme of the above dissection of $\mathbf{P}_i\cap (\mathbf{c}_i+\sqrt{2}\mathbf{B}^3)$ can be obtained as a limit of proper $3$-dimensional orthoschemes therefore one can use the method of limiting density exactly the same way as it is described in \cite{Bo} to obtain the following conclusion.

\begin{Sublemma}\label{Rogers-Bezdek}
$$
\frac{{\rm vol}_3\left(  (\mathbf{P}_i\cap (\mathbf{c}_i+\sqrt{2}\mathbf{B}^3))   \cap (\mathbf{c}_i+\mathbf{B}^3)\right)}{{\rm vol}_3(\mathbf{P}_i\cap (\mathbf{c}_i+\sqrt{2}\mathbf{B}^3))}=
\frac{\frac{4\pi}{3}}{{\rm vol}_3(\mathbf{P}_i\cap (\mathbf{c}_i+\sqrt{2}\mathbf{B}^3))} \le \frac{{\rm vol}_3(\mathbf{W}^3\cap\mathbf{B}^3)}{{\rm vol}_3(\mathbf{W}^3)}<0.6401.$$
\end{Sublemma} 

Finally, as $\mathbf{P}_i\cap (\mathbf{c}_i+\sqrt{2}\mathbf{B}^3)\subset \mathbf{P}_i\cap (\mathbf{c}_i+\sqrt{3}\mathbf{B}^3)$, therefore Sublemma~\ref{Rogers-Bezdek} completes the proof of Lemma~\ref{3D density upper bound}.
\end{proof}

The well-known isoperimetric inequality (\cite{Os78}) applied to $\bigcup_{i=1}^n\left(\mathbf{c}_i+\sqrt{3}\mathbf{B}^3\right)$ yields 

\begin{Lemma}\label{isoperimetric-inequality}
$$36\pi\ {\rm vol}_3\left(\bigcup_{i=1}^n\left(\mathbf{c}_i+\sqrt{3}\mathbf{B}^3\right)\right)^2
\le{\rm svol}_2\left({\rm bd}\left(\bigcup_{i=1}^n\left(\mathbf{c}_i+\sqrt{3}\mathbf{B}^3\right)\right)\right)^3,$$
where ${\rm svol}_2(\cdot)$ refers to the $2$-dimensional surface volume of the corresponding set.
\end{Lemma}

Thus, Lemma~\ref{3D density upper bound} and Lemma~\ref{isoperimetric-inequality} generate the following inequality.

\begin{Corollary}\label{lower-bound-for-surface-area-in-3D}
$$\frac{4\pi}{(0.6401)^{\frac{2}{3}}}n^{\frac{2}{3}}< {\rm svol}_2\left({\rm bd}\left(\bigcup_{i=1}^n\left(\mathbf{c}_i+\sqrt{3}\mathbf{B}^3\right)\right)\right).$$
\end{Corollary}

Now, assume that $\mathbf{c}_i+\mathbf{B}^3\in {\cal P}$ is tangent to $\mathbf{c}_j+\mathbf{B}^3\in {\cal P}$ for all $j\in T_i$, where $T_i\subset\{1, 2, \dots , n\}$ stands for the family of indices $1\le j\le n$ for which ${\rm dist}(\mathbf{c}_i, \mathbf{c}_j)=2$. Then let
$\hat{S}_i:={\rm bd}(\mathbf{c}_i+\sqrt{3}\mathbf{B})$ and let $\hat{\mathbf{c}}_{ij}$ be the intersection of the line segment $\mathbf{c}_i\mathbf{c}_j$ with $\hat{S}_i$ for all $j\in T_i$. Moreover, let $C_{\hat{S}_i}(\hat{\mathbf{c}}_{ij}, \frac{\pi}{4})$ (resp., $C_{\hat{S}_i}(\hat{\mathbf{c}}_{ij}, \alpha)$) denote the open spherical cap of $\hat{S}_i$ centered at $\hat{\mathbf{c}}_{ij}\in \hat{S}_i$ having angular radius $\frac{\pi}{4}$ (resp., $\alpha$ with $0<\alpha<\frac{\pi}{2}$ and $\cos\alpha=\frac{1}{\sqrt{3}}$). As ${\cal P}$ is totally separable therefore the family $\{C_{\hat{S}_i}(\hat{\mathbf{c}}_{ij}, \frac{\pi}{4}), j\in T_i\}$ consists of pairwise disjoint open spherical caps of $\hat{S}_i$; moreover,

\begin{equation}\label{Bezdek-estimate-I}
\frac{\sum_{j\in T_i}{\rm svol}_2\left(C_{\hat{S}_i}(\hat{\mathbf{c}}_{ij}, \frac{\pi}{4})\right)}{{\rm svol}_2\left(\cup_{j\in T_i}C_{\hat{S}_i}(\hat{\mathbf{c}}_{ij}, \alpha)\right)}=
\frac{\sum_{j\in T_i}{\rm Sarea}\left(C(\mathbf{u}_{ij}, \frac{\pi}{4})\right)}{{\rm Sarea}\left(\cup_{j\in T_i}C(\mathbf{u}_{ij}, \alpha)\right)},
\end{equation}

\noindent where $\mathbf{u}_{ij}:=\frac{1}{2}(\mathbf{c}_j-\mathbf{c}_i)\in \mathbb{S}^2:={\rm bd}(\mathbf{B}^3)$ and $C(\mathbf{u}_{ij}, \frac{\pi}{4})\subset \mathbb{S}^2$ (resp., $C(\mathbf{u}_{ij}, \alpha)\subset \mathbb{S}^2$) denotes the open spherical cap of $\mathbb{S}^2$ centered at $\mathbf{u}_{ij}$ having angular radius $\frac{\pi}{4}$ (resp., $\alpha$)
and where ${\rm Sarea}(\cdot)$ refers to the spherical area measure on $\mathbb{S}^2$. 

\begin{Lemma}\label{spherical density upper bound}
$$\frac{\sum_{j\in T_i}{\rm Sarea}\left(C(\mathbf{u}_{ij}, \frac{\pi}{4})\right)}{{\rm Sarea}\left(\cup_{j\in T_i}C(\mathbf{u}_{ij}, \alpha)\right)}\le 3\left(1-\frac{1}{\sqrt{2}}\right)=0.8786\ldots\ .$$
\end{Lemma}

\begin{proof}
By assumption ${\cal P}_i(\mathbb{S}^2)=\{C(\mathbf{u}_{ij}, \frac{\pi}{4})\ |\  j\in T_i\}$ is a packing of spherical caps of angular radius $\frac{\pi}{4}$ in $\mathbb{S}^2$.
Let $V_{ij}(\mathbb{S}^2)$ denote the Voronoi region of the packing ${\cal P}_i(\mathbb{S}^2)$ assigned to $C(\mathbf{u}_{ij}, \frac{\pi}{4})$,  that is, let $V_{ij}(\mathbb{S}^2)$ stand for the set of points of $\mathbb{S}^2$ that are not farther away from $\mathbf{u}_{ij}$ than from any other $\mathbf{u}_{ik}$ with $k\neq j, k\in T_i$. Recall (see for example \cite{FT64}) that the Voronoi regions $V_{ij}(\mathbb{S}^2)$, $j\in T_i$ are spherically convex polygons and form a tiling of $\mathbb{S}^2$. Moreover, it is easy to see that no vertex of $V_{ij}(\mathbb{S}^2)$ belongs to the interior of $C(\mathbf{u}_{ij}, \alpha)$ in $\mathbb{S}^2$. Thus, Haj\'os Lemma (Hilfssatz 1 in \cite{Mo65}) implies that ${\rm Sarea}\left(V_{ij}(\mathbb{S}^2)\cap C(\mathbf{u}_{ij}, \alpha)\right)\ge\frac{2\pi}{3}$ because $\frac{2\pi}{3}$ is the spherical area of a regular spherical quadrilateral inscribed into $C(\mathbf{u}_{ij}, \alpha)$ with sides tangent to $C(\mathbf{u}_{ij}, \frac{\pi}{4})$. Hence, 
\begin{equation}\label{spherical local estimate}
\frac{{\rm Sarea}\left(C(\mathbf{u}_{ij}, \frac{\pi}{4})\right)}{{\rm Sarea}\left(V_{ij}(\mathbb{S}^2)\cap C(\mathbf{u}_{ij}, \alpha)\right)}\le3\left(1-\frac{1}{\sqrt{2}}\right)\ .
\end{equation}
As the truncated Voronoi regions $V_{ij}(\mathbb{S}^2)\cap C(\mathbf{u}_{ij}, \alpha)$, $j\in T_i$ form a tiling of $\cup_{j\in T_i}C(\mathbf{u}_{ij}, \alpha)$ therefore (\ref{spherical local estimate}) finishes the proof of Lemma~\ref{spherical density upper bound}.
\end{proof}

Lemma~\ref{spherical density upper bound} implies in a straightforward way that

\begin{equation}\label{Bezdek-estimate-II}
{\rm svol}_2\left({\rm bd}\left(\bigcup_{i=1}^n\left(\mathbf{c}_i+\sqrt{3}\mathbf{B}^3\right)\right)\right)\le 12\pi n-\frac{1}{3\left(1-\frac{1}{\sqrt{2}}\right)}12\pi \left(1-\frac{1}{\sqrt{2}}\right)c(n,3)=12\pi n-4\pi c(n,3).
\end{equation}

Hence, Corollary~\ref{lower-bound-for-surface-area-in-3D} and (\ref{Bezdek-estimate-II}) yield

$$\frac{4\pi}{(0.6401)^{\frac{2}{3}}}n^{\frac{2}{3}}<12\pi n-4\pi c(n,3),$$
from which it follows that $c(n,3)<3n-\frac{1}{(0.6401)^{\frac{2}{3}}}n^{\frac{2}{3}}<3n-1.346n^{\frac{2}{3}}$, finishing the proof of Theorem~\ref{3D}.

\section{Appendix}

We use the method of Harborth \cite{Ha} with some natural modifications due to the total separability of the packings under investigation. We prove (\ref{H-0}) by induction on $n$. For simplicity let $c(n):=c(n,2)$. Clearly, $c(2)=1=\lfloor 2\cdot 2- 2\sqrt{2}\rfloor$. So in what follows, we assume that $n\ge 3$ and in particular, we assume that (\ref{H-0}) holds for all positive integers $n'$ with $2\le n'\le n-1$. Let ${\cal P}_n$ be the totally separable packing of $n$ unit disks in $\mathbb{E}^{2}$, which has the largest number namely, $c(n)$ of touching pairs among all totally separable packings of $n$ unit disks in $\mathbb{E}^2$. (${\cal P}_n$ might not be uniquely determined up to congruence in which case ${\cal P}_n$ stands for any of those extremal packings.) Let $G_n$ denote the embedded contact graph of ${\cal P}_n$ with vertices identical to the centers of the unit disks in ${\cal P}_n$ and with edges represented by line segments connecting two vertices if the unit disks centered at them touch each other. Clearly, the number of edges of $G_n$ is equal to $c(n)$. As $c(n-1)+1=\lfloor 2(n-1)- 2\sqrt{n-1}\rfloor+1\le \lfloor 2n- 2\sqrt{n}\rfloor$ and $c_{\mathbb{Z}}(n,2)=\lfloor 2n- 2\sqrt{n}\rfloor $ (\cite{HaHa}) for all $n\ge 2$, therefore one can assume that every vertex of $G_n$ is adjacent to at least two other vertices (otherwise there is a vertex of $G_n$ of degree one and so, the proof is finished by induction). In addition, using $c_{\mathbb{Z}}(n,2)=\lfloor 2n- 2\sqrt{n}\rfloor $ again one can assume that $G_n$ is $2$-connected, that is, $G_n$ remains connected after the removal of any of its vertices.

Thus, the outer face of $G_n$ in $\mathbb{E}^{2}$ is bounded by a simple closed polygon $P$. Let $b$ denote the number of vertices of $P$. As ${\cal P}_n$ is a totally separable unit disk packing therefore the degree of any vertex of $P$ (resp., $G_n$) is either $2$ or $3$ or $4$ in $G_n$. Let $b_i$ stand for the number of vertices of $P$ of degree $i$ with $2\le i\le 4$. Clearly, $b=b_2+b_3+b_4$. Due to the total separability of ${\cal P}_n$, the internal angle of $P$ at a vertex of degree $i$ is at least $\frac{(i-1)\pi}{2}$, and the sum of these angles is $(b-2)\pi$. Thus,
\begin{equation}\label{H-1}
b_2+2b_3+3b_4\le 2b-4
\end{equation}
Next, let $f_i$ denote the number of internal faces of $G_n$ having $i$ sides. As ${\cal P}_n$ is totally separable therefore $i\ge 4$. Now, Euler's formula implies that
\begin{equation}\label{H-2}
n-c(n)+f_4+f_5+ {\text \dots} =1
\end{equation}
If we add up the number of sides of the internal faces of $G_n$, then every edge of $P$ is counted once and all the other edges twice. Thus,
\begin{equation}\label{H-3}
4(f_4+f_5+{\text \dots})\le 4f_4+5f_5+{\text \dots}=b+2(c(n)-b).
\end{equation}
Clearly, (\ref{H-2}) and (\ref{H-3}) imply that $4(1-n+c(n))\le b+2(c(n)-b)$ and so, 
\begin{equation}\label{H-4}
 2c(n)-3n+4\le n-b
\end{equation}
Now, let us delete from $G_n$ the vertices of $P$ together with the edges incident to them. By the definition of $c(n-b)$, one obtains
\begin{equation}\label{H-5}
c(n)-b-(b_3+2b_4)\le c(n-b).
\end{equation}
Next, (\ref{H-1}) and (\ref{H-5}) imply
\begin{equation}\label{H-6}
c(n)\le c(n-b)+2b-4.
\end{equation}
As by induction $c(n-b)\le2(n-b)-2\sqrt{n-b}$, therefore (\ref{H-6}) yields
\begin{equation}\label{H-7}
c(n)\le (2n-4)-2\sqrt{n-b}.
\end{equation}
Finally, (\ref{H-4}) and (\ref{H-7}) imply $c(n)\le (2n-4)-2\sqrt{2c(n)-3n+4}$, from which it follows easily that
\begin{equation}\label{H-8}
0\le c(n)^2-4nc(n)+(4n^2-4n).
\end{equation} 
Notice that the roots of the quadratic equation $0=x^2-4nx+(4n^2-4n)$ are $2n\pm2\sqrt{n}$. As $c(n)<2n$, therefore (\ref{H-8}) implies in a straightforward way that $c(n)\le 2n-2\sqrt{n}$, finishing the proof of (\ref{H-0}).

\vspace{1cm}

\medskip

\noindent
K\'aroly Bezdek
\newline
Department of Mathematics and Statistics, University of Calgary, Canada,
\newline
Department of Mathematics, University of Pannonia, Veszpr\'em, Hungary,
\newline
{\sf E-mail: bezdek@math.ucalgary.ca}
\vskip0.5cm

\noindent Bal\'azs Szalkai
\newline
Institute of Mathematics, E\"otv\"os University, Budapest, Hungary, 
\newline
{\sf E-mail: szalkai@pitgroup.org}
\vskip0.5cm
\noindent and
\vskip0.5cm

\noindent Istv\'an Szalkai
\newline
Department of Mathematics, University of Pannonia, Veszpr\'em, Hungary,
\newline
{\sf E-mail: szalkai@almos.uni-pannon.hu}

\end{document}